\documentclass[12pt,reqno]{amsart}
\usepackage[shortlabels]{enumitem}
\usepackage{esint}
\usepackage{physics}
\usepackage{relsize}
\usepackage[backref]{hyperref}
\usepackage{mathtools}
\usepackage{amsfonts}
\usepackage[all]{xy}
\usepackage{geometry}
\usepackage{amsmath,amssymb} 
\usepackage{faktor}

\newtheorem{theorem}{Theorem}[section]

\theoremstyle{corollary}

\theoremstyle{conjecture}

\theoremstyle{assumption}

\theoremstyle{proposition}
\newtheorem{proposition}[theorem]{Proposition}
\theoremstyle{remark}
\newtheorem{remark}[theorem]{Remark}
\numberwithin{equation}{section}

\newcommand{\ii}{\ensuremath{\sqrt{-1}}}
\newcommand{\pp}{\bar\partial}
\everymath{\displaystyle}

\DeclareMathOperator{\Vol}{dVol}

\DeclareMathOperator{\Arg}{\Theta}

\begin{document}

\title[]{A note on the supercritical deformed Hermitian-Yang-Mills equation}
\author{Junsheng Zhang}
\address{Department of Mathematics\\
 University of California\\
 Berkeley, CA, USA, 94720\\}
\curraddr{}
\email{jszhang@berkeley.edu}
\thanks{}


\maketitle
\begin{abstract}
We show that on a compact K\"ahler manifold all real $(1,1)$-classes admitting solutions to the supercritical deformed Hermitian-Yang-Mills equation form a both open and closed subset of those which satisfy the numerical condition proposed by Collins-Jacob-Yau. Moreover we show by examples that it can be a proper subset. This disproves a conjecture made by Collins-Jacob-Yau. 
\end{abstract}

\section{Introduction}\label{s1}
Let $(X,\chi)$ be an $n$-dimensional compact K\"ahler manifold, and $\omega$ be a closed real $(1,1)$-form. Let $\Theta_{\chi}(\omega)\in (-\pi,\pi]$ denote the principal argument of the complex number $ \int_X(\omega+\ii\chi)^n,$ which we always assume to be nonzero.
Let $\omega_{\varphi}$ denote $\omega+\ii\partial\pp \varphi$ for any smooth real function $\varphi$ on $X$.  Then the supercritical deformed Hermitian-Yang-Mills (dHYM) equation is given by 
\begin{equation}\label{specified lagrangian}
    Q_{\chi}(\omega_{\varphi}):=\sum_{i=1}^n\arccot(\lambda_i)=\Arg_{\chi}(\omega)\in (0,\pi),
\end{equation} where $\lambda_i$ denotes the eigenvalues of $\omega_{\varphi}$ with respect to $\chi$ and $\arccot(x)=\frac{\pi}{2}-\arctan(x)$ takes values in $(0,\pi)$. 


For any closed real $(1,1)$-form $\omega$ and real number $\theta \in (0,\pi)$, let  $\Vol_{\chi}^p(\omega,\theta)$ denote the closed real $(p,p)$-form 
\begin{equation}\label{form}
    \operatorname{Re}(\omega+\sqrt{-1} \chi)^p-\cot (\theta) \operatorname{Im}(\omega+\sqrt{-1} \chi)^p.
\end{equation} Moreover if $\theta$ coincides with $\Arg_{\chi}(\omega)$, then we will omit the dependence on $\theta$ and just denote the form \eqref{form} by $\Vol_{\chi}^p(\omega)$.
The following result follows from \cite[Lemma 8.2]{CJY}, and for completeness, we include a short proof in Section \ref{s2}.

\begin{proposition}[\cite{CJY}]\label{easy direction}
If $\omega$ is a solution to \eqref{specified lagrangian}, i.e. $ Q_{\chi}(\omega)=\Arg_{\chi}(\omega)\in (0,\pi),$ then for any p-dimensional analytic subvariety $V$ of $X$, $0<p<n$, we have 
\begin{equation}\label{conj}
    \int_V\Vol_{\chi}^p(\omega)>0.
\end{equation}
\end{proposition}

Note that the integral $ \int_V\Vol_{\chi}^p(\omega)$ only depends on the cohomology classes $[\chi],[\omega]\in H^{1,1}_{\mathbb R}(X)$ for any $p$-dimensional analytic subvariety $V$, $p\geq 1$ and $\Arg_{\chi}(\omega)$ also only depends on the cohomological data.
Then we can introduce some notations: 
\begin{equation}
\begin{aligned}
 \mathcal{P}_{\chi} &= \left\{[\omega]\in H^{1,1}_{\mathbb R}(X): \Arg_{\chi}(\omega)\in (0,\pi)  \text{ and $\int_V\Vol_{\chi}^p(\omega)>0$ for any }\right. \\
 &\hspace{1.3in}\left. {\text{ $p$-dimensional analytic subvariety $V$, $0<p<n$}} \right\}, \\
  \mathcal{K}_{\chi}&:=\left\{[\omega]\in H^{1,1}_{\mathbb R}(X): \eqref{specified lagrangian} \text{ admits a smooth solution}\right\}.
 \end{aligned}
\end{equation} By the definition it is clear that $\mathcal{P}_{\chi}$ depends only on the class $[\chi]\in H^{1,1}_{\mathbb R}(X)$, but not clear whether $\mathcal{K}_{\chi}$ depends only on $[\chi]$. Although not needed in this paper, we mention that by the remarkable result \cite[Theorem 1.7]{Chen}, $\mathcal{K}_{\chi}$ indeed depends only on $[\chi]$.

Proposition \ref{easy direction} shows that $\mathcal{K}_{\chi}\subset \mathcal{P}_{\chi}$ and Collins-Jacob-Yau conjectured  \cite[Conjecture 1.5]{CJY} that for any compact K\"ahler manifold $(X,\chi)$, 
\begin{equation}\label{CJY conj}
\mathcal{K}_{\chi}=\mathcal{P}_{\chi}.
\end{equation}
  This has been confirmed for compact K\"ahler surfaces \cite{CJY} and projective manifolds \cite{Chen,DP,CLT,bal}.
  In this paper, we present examples that demonstrate $\mathcal{K}_{\chi}\neq \mathcal{P}_{\chi}$, disproving the conjecture for general compact K\"ahler manifolds. Initially, the author's examples involved blowing up a point on a generic torus. However, as Mao Sheng pointed out to the author, counterexamples could still exist, even for generic tori. Our main result is as follows.

\begin{theorem}\label{main theorem}
   For any compact K\"ahler manifold $(X,\chi)$,  $\mathcal{K}_{\chi}$ is a both open and closed subset of $\mathcal{P}_{\chi}$. Moreover for any $n\geq 3$, there exist $n$-dimensional compact K\"ahler manifolds for which  $\mathcal K_{\chi} \neq \mathcal{P}_{\chi}$.
\end{theorem}

\subsection*{Acknowledgements}  The author expresses his gratitude to his advisor Song Sun for the constant support and invaluable suggestions. He would like to thank Vamsi Pritham Pingali for pointing out the reference \cite{bal} and Yifan Chen, Jianchun Chu and Liding Huang for their interest in the work. The author extends special thanks to Jianchun Chu for his helpful comments and for pointing out references \cite{CY} and \cite{CL}. He is grateful to Mao Sheng for suggesting much easier counterexamples of \eqref{CJY conj}. He thanks the anonymous referee for pointing out an imprecise statement in the previous version of the paper and suggestions improving the presentation.

\section{proof}\label{s2}
\noindent\textit{Proof of Proposition \ref{easy direction}.} Let $\omega$ be a solution of the supercritical dHYM equation \eqref{specified lagrangian}.
    Let $\chi^{-1}\cdot \omega$ denote the endomorphism of $T^{1,0}X$ defined by 
    \begin{equation*}
        \omega(\cdot,\cdot)=\chi(K\cdot,\cdot).
    \end{equation*} The notation $(\left.\chi\right|_{V})^{-1}\cdot(\left.\omega\right|_{V})$ can be understood similarly. Then locally in a neighborhood of any regular point of $V$, we can choose holomorphic coordinates such that $\chi^{-1}\cdot \omega$ is given by a Hermitian matrix $A$ and $(\left.\chi\right|_{V})^{-1}\cdot(\left.\omega\right|_{V})$ is given by a hermitian matrix $B$, which is a principal submatrix of $A$, i.e.
    \begin{equation*}
        A=\left(\begin{array}{cc}
           B  & C \\
           C^*  & D
        \end{array}\right)
    \end{equation*} where $*$ denotes the conjugate transpose. Suppose eigenvalues of $A$ are $\lambda_1\geq \cdots\geq\lambda_k\geq \cdots \geq \lambda_n$ and the eigenvalues of $B$ are $\mu_1\geq \cdots\geq\mu_k\geq \cdots\geq \mu_p$. Then by the min-max theorem in linear algebra, we have
    \begin{equation}\label{eigenvalue comparison}
        \lambda_j \geqslant \mu_j \geqslant \lambda_{j+n-p}, \quad j=1,2, \cdots, p.
    \end{equation}
When restricted to the regular part of $V$, we have
\begin{equation}\label{restricted volume}
\begin{aligned}
\left.\Vol_{\chi}^p(\omega)\right|_V&=\left(\Re\prod_{i=1}^p(\mu_i+\ii)-\cot(\Arg_{\chi}(\omega))\Im\prod_{i=1}^p(\mu_i+\ii)\right) \left.\chi^p\right|_V\\
     &=\left(\cot(\sum_{i=1}^p\arccot(\mu_i))-\cot(\Arg_{\chi}(\omega))\right)\sin(\sum_{i=1}^p\arccot(\mu_i))\left.\chi^p\right|_V.
\end{aligned}
\end{equation}
Since $\omega$ is a solution of $\eqref{specified lagrangian}$, we have $\sum_{i=1}^n\arccot(\lambda_i)=\Arg_{\chi}(\omega)\in (0,\pi)$. Note that the function $\arccot$ is decreasing and $p<n$, then by $\eqref{eigenvalue comparison}$ we know
    \begin{equation}
        0<\sum_{i=1}^p\arccot(\mu_i)\leq \sum_{i=n-p+1}^n\arccot(\lambda_i)<\sum_{i=1}^n\arccot(\lambda_i)< \pi.
    \end{equation}Therefore the right hand side of \eqref{restricted volume} is positive pointwisely on the regular part of $V$ and as a consequence we get the integral of  $\Vol_{\chi}^p(\omega)$ over $V$ is positive.
$\hfill \qed$

\begin{remark}\label{positve along test family}	As observed in \cite{Chen}, one can indeed obtain that if $\eqref{specified lagrangian}$ admits a smooth solution, then for $\omega_t=\omega+t\chi$, $ \int_V\Vol_{\chi}^p\left(\omega_t,\Arg_{\chi}(\omega)\right)$ is non-decreasing with respect to $t$ on $[0,\infty)$ for any $p$-dimensional analytic subvariety $V$, $1\leq p\leq n$
	
\end{remark}

In the following, we also use notations $\int_V\Vol^p_{\chi}(\beta)$ and $\Arg_{\chi}(\beta)$ for $\beta\in H^{1,1}_{\mathbb R}(X)$, where $V$ is a $p$-dimensional analytic subvariety. Their definitions are clear from the discussion in Section \ref{s1}.\\

\noindent\textit{Proof of Theorem \ref{main theorem}.} 
  The first statement essentially follows form the results of \cite{CJY} and \cite{CLT}.  We need to show that $\mathcal{K}_{\chi}$ is both open and closed in $\mathcal P_{\chi}$.
  
  We firstly show that $\mathcal{K}_{\chi}$ is an open subset of $H^{1,1}_{\mathbb{R}}(X)$. This is a standard application of the implicit function theorem. See for example \cite[Lemma 7.2]{CJY}. For readers' convenience, we include a proof here. Let $[\omega_0]\in \mathcal{K}_{\chi}$. We may assume $\omega_0$ is harmonic with respect to $\chi$ and by choosing harmonic representatives, we may identify a neighborhood of $[\omega_0]$ in $H^{1,1}_{\mathbb R}(X)$ with a subset set $U$ of smooth closed real $(1,1)$-forms containing $\omega_0$. Since $[\omega_0]\in \mathcal{K}_{\chi}$, there is a smooth function $\varphi_0$ such that $Q_{\chi}(\omega_0+\ii\partial\pp \varphi_0)=\Arg_{\chi}(\omega_0)\in (0,\pi)$. Then fix $\beta\in (0,1)$, $k\geq 2$ and consider the map
     $ F:(-1,1)\times U\times {C^{k,\beta}(X)}/{\mathbb {R}}\rightarrow C^{k-2,\beta}(X)$
given by 
\begin{equation}
    F(s,\alpha,\varphi)= s+Q_{\chi}(\alpha+\ii \partial\pp \varphi)
\end{equation}
Then the linearization of $F$ along the first and the third components at the point $(0,\omega_0,\varphi_0)$ is given by 
\begin{equation*}
    (s,v)\rightarrow s+\Delta_{\eta}(v),
\end{equation*} which is an isomorphism from $\mathbb R\times C^{k,\beta}(X)/\mathbb R$ to  $C^{k-2,\beta}(X)$. Here the operator $\Delta_{\eta}$ is given by \cite[Section 2]{CJY}
\begin{equation*}
    \Delta_\eta=\eta^{j \bar{k}} \partial_j \partial_{\bar{k}},
\end{equation*} where $\eta_{j\bar k}=\chi_{j \bar{k}}+\omega_{j \bar{l}} \chi^{p \bar{l}} \omega_{p \bar{k}}$ and $\omega=\omega_0+\ii\partial\pp \varphi_0$.  By the implicit function theorem we know that there exists a neighborhood $W\subset U$ of $\omega_0$ such that, for any $\alpha\in W$ we can find a unique pair $(t,\varphi)\in \mathbb R\times C^{k,\beta}(X)/\mathbb R$ in a neighborhood of $(0,\varphi_0)$ satisfying
\begin{equation}\label{equaiton for alpha}
    Q_{\chi}(\alpha+\ii\partial\pp \varphi)=\Arg_{\chi}(\omega_0)-t.
\end{equation} Integrating over $X$, we obtain that the constant on the right hand side of \eqref{equaiton for alpha} has to be $\Arg_{\chi}(\alpha)$. By a standard boot strapping argument, we know that $\varphi$ is in fact a smooth function. Therefore $[\alpha]\in \mathcal{K}_{\chi}$ for any $\alpha\in W$, i.e. a neighborhood of $[\omega_0]$ in $H^{1,1}_{\mathbb R}(X)$ is contained in $\mathcal{K}_{\chi}$.
  

  Secondly we show that $\mathcal{K}_{\chi}$ is closed in $\mathcal{P}_{\chi}$, i.e. that $\overline{\mathcal{K}_{\chi}}\cap \mathcal{P}_{\chi}\subset \mathcal{K}_{\chi}$. Let $\alpha\in \overline{\mathcal{K}_{\chi}}\cap \mathcal{P}_{\chi} $. In particular, we have $\Arg_{\chi}(\alpha)\in (0,\pi)$. In order to show $\alpha\in \mathcal{K}_{\chi}$, by \cite[Theorem 1.3]{CLT} it is sufficient to show that for $\alpha_t=\alpha+t[\chi]$, for all $t \in [0,\infty)$ and 
for any $p$-dimensional analytic subvariety $V$, we have:
\begin{equation}
    \int_V\Vol_{\chi}^p\left(\alpha_t,\Arg_{\chi}(\alpha)\right)\geq 0
\end{equation} and the strict inequality holds if $p < n $ for all $t \in [0,\infty)$. Let us prove these inequalities. Since $\alpha \in \overline{\mathcal{K}_{\chi}}$, there exists a sequence $\alpha_i$ in $\mathcal{K}_{\chi}$ converges to $\alpha$. Let $\alpha_{i,t}$ denote $\alpha_i+t[\chi]$. Then for any fixed $t\in [0,\infty)$ and $p$-dimensional analytic subvariety $V$, we have
\begin{equation}
    \begin{aligned}
\int_V\Vol_{\chi}^p\left(\alpha_t,\Arg_{\chi}(\alpha)\right)&=\lim_{i\rightarrow \infty} \int_V\Vol_{\chi}^p\left(\alpha_{i,t},\Arg_{\chi}(\alpha)\right)=\lim_{i\rightarrow \infty} \int_V\Vol_{\chi}^p\left(\alpha_{i,t},\Arg_{\chi}(\alpha_i)\right)\\
&\geq \lim_{i\rightarrow \infty}   \int_V\Vol_{\chi}^p\left(\alpha_i,\Arg_{\chi}(\alpha_i)\right)= \int_V\Vol_{\chi}^p(\alpha)\geq 0
    \end{aligned}
\end{equation} where the first inequality in the second row follows from Remark \ref{positve along test family} and the last inequality is strict if $p < n $ since $\alpha \in \mathcal{P}_{\chi}$.
\vspace{1cm}

Let us elaborate on why, in the case where a torus admits no proper analytic subvarieties, we have $\mathcal{K}_{\chi}\neq \mathcal{P}_{\chi}$. Consider an $n$-dimensional complex torus $X=\mathbb C^n/\Lambda$, where $(z_1,\cdots,z_n)$ represents the standard coordinate on $\mathbb C^n$. Let  $\chi=\ii \sum_{i=1}^{n}dz_i\wedge d\bar z_i$ and  $\omega=A \chi$, where $A\in \mathbb R$. Apparently the equation
\begin{equation*}
	\mathcal Q_{\chi}(\omega_{\varphi})=n\arccot(A)
\end{equation*}
admits a smooth solution. The maximal principle ensures that if there exists a constant $c$ such that $\mathcal Q_{\chi}(\omega_{\varphi})=c$ has a solution, then this constant $c$ must be $n\arccot(A)$. Therefore we know that 
\begin{equation*}
	A\chi\in \mathcal{K}_{\chi} \text{ if and only if } n\arccot(A)\in (0,\pi).
\end{equation*} On the other hand, given the torus's property of having no positive dimensional proper analytic subvarieties, we know that 
\begin{equation*}
	A\chi\in \mathcal{P}_{\chi} \text{ if and only if } n\arccot(A) \text{ mod } 2\pi \in (0,\pi).
\end{equation*}
Then it follows that if $n\geq 3$, $\mathcal{K}_{\chi}\neq \mathcal{P}_{\chi}$.
$\hfill \qed$

As one can see, the key issue lies in the definition of the supercritical deformed Hermitian-Yang-Mills equations. We require not only $\Arg_{\chi}(\omega)\in (0,\pi)$ but also pointwisly $\mathcal Q_{\chi}(\omega_{\varphi})\in (0,\pi)$. However when $X$ has no positive dimensional proper analytic subvarieties, $\omega\in \mathcal P_{\chi}$ only requires $\Arg_{\chi}(\omega)\in (0,\pi)$.
To address this issue, one might consider refining the definition of the set $\mathcal K_{\chi}$, giving rise to a new set denoted as $\mathcal K^1_{\chi}$. Specifically, we may define $\omega \in \mathcal K^1_{\chi}$ if $\Arg_{\chi}(\omega)\in (0,\pi)$ and the equation
\begin{equation*}
	\mathcal Q_{\chi}(\omega_{\varphi})=\Arg_{\chi}(\omega) \text{ mod } 2\pi
	\end{equation*}admits a solution.
However, as noted by G.Chen \cite[Remark 1.10]{Chen}, $\mathcal K^1_{\chi}$ does not necessarily belong to $\mathcal P_{\chi}$.
Therefore it seems that except for some special cases considered in \cite[Section 8]{CY} and  \cite{CL}, in order to give a numerical criterion for $\omega\in \mathcal K_{\chi}$ on a K\"ahler (non-projective) manifold, the test family introduced by G.Chen \cite[Theorem 1.7]{Chen} is necessary.

\section{Discussion}
As mentioned above, not much is known for numerical conditions on non-supercritical dHYM equations.
Here we present a general numerical obstruction for dHYM equations on 3-dimensional compact K\"ahler manifolds, which has been essentially proven by Collins-Xie-Yau in \cite{CXY}. 
 
The author acknowledges Jianchun Chu for pointing out the reference \cite{CY}, which covers most of the material presented in the initial version of this section. We therefore refer to \cite[Section 8]{CY} for further discussion on dHYM equations on 3-dimensional K\"ahler manifolds.
 As a reminder, $Q_{\chi}(\omega_{\varphi})=\sum_{i=1}^n\arccot(\lambda_i)$, where $\lambda_i$ denotes the eigenvalues of $\omega_{\varphi}$ with respect to $\chi$, and $\arccot(x)=\frac{\pi}{2}-\arctan(x)$ takes values in $(0,\pi)$.
\begin{proposition}\label{prop nonzero}
	On a compact 3-dimensional K\"ahler manifold $(X,\chi)$, suppose
	$Q_{\chi}(\omega_{\varphi})=\theta$ admits a smooth solution for some constant $\theta\in (0,3\pi)$. Then for any $t\in [0,1]$, we have
	\begin{equation}\label{nonzero}
		\gamma([\omega],t):=\int_X(t\omega+\ii \chi)^3\neq 0.
	\end{equation}
\end{proposition}

\begin{proof}
	Note that $\gamma([\omega],t)=0$ if and only if $\gamma(-[\omega],t)=0$. Therefore, we can assume without loss of generality that $Q_{\chi}(\omega_{\varphi}) =\theta\in (0,\frac{3\pi}{2}]$, by replacing $[\omega]$ with $-[\omega]$ if necessary.
	\begin{itemize}
		\item Case 1, $\theta\in (0,\pi)$. This is due to Collins-Xie-Yau \cite[Proposition 3.3]{CXY} and indeed they proved the following Chern number inequaly which implies \eqref{nonzero}.
		\begin{equation}\label{chern number ineq}
\int_X\omega^3\cdot\int_X\chi^3
<9\int_X\omega^2\wedge\chi\cdot\int_X\omega\wedge\chi^2.	\end{equation}

		\item Case 2, $\theta\in [\pi,\frac{3\pi}{2})$. Then the problem is much easier. Observe that
			$$\int_X\Im(t\omega+\ii\chi)^3=3t^2\int_X\omega^2\wedge\chi-\int_X\chi^3$$
  is a monotonic function of $t$ and is negative at $t=0$. As the condition $\theta\in [\pi,\frac{3\pi}{2})$ implies that the function is non-positive at $t=1$, it follows that it must be negative, and in particular nonzero, if $t\in [0,1)$.
	\end{itemize}
\end{proof}

Note that the numerical obstruction \eqref{nonzero} holds not only for supercritical dHYM equations, that is there is no restriction  on the angle. In particular, if there is a class $\alpha$ with $\int_X(\alpha+\ii \chi)^3=0$, then for any $t>1$, there is no smooth function $\varphi$ such that $Q_{\chi}(t\alpha_{\varphi})$ is a constant, regardless of the chosen constant.

As discussed in \cite{CXY,CY}, the crucial point is that if \eqref{nonzero} holds, a unique smooth function $\Theta([\omega], \cdot):[0,1]\rightarrow \mathbb R$ exists, with $\Theta([\omega],0)=\frac{3}{2}\pi$ and $\Im\left(e^{-\ii \Theta([\omega],t)}\gamma([\omega],t)\right)=0$ for all $t\in [0,1]$. It is easy to show that if $Q_{\chi}(\omega_{\varphi})=\theta$ admits a smooth solution and $\Theta([\omega],t)$ is the aforementioned function, then $\Theta([\omega],1)=\theta$. Therefore if \eqref{nonzero} holds, it enables a priori determination of the constant in the dHYM equations and by requiring $\Theta([\omega],1)\in (0,\pi)$, we can effectively address the challenge posed by the modulo $2\pi$ issue, which plays an essential role in the counterexamples we discussed. 

To obtain a numerical characterization of $\mathcal K_{\chi}$ (without using the test family introduced by G. Chen), further conditions are necessary, as discussed in \cite[Section 8]{CY}. It is worth noting that one of the conjectures proposed in that section has been proven in \cite{CL}, which builds on the work in \cite{CLT}.

\bibliographystyle{plain}
\bibliography{ref.bib}

\begin{thebibliography}{1}

\bibitem{bal}
Aashirwad Ballal.
\newblock {The supercritical deformed Hermitian Yang--Mills equation on compact
  projective manifolds}.
\newblock {\em Illinois Journal of Mathematics}, 1(1):1--27, 2023.

\bibitem{Chen}
Gao Chen.
\newblock {The J-equation and the supercritical deformed Hermitian--Yang--Mills
  equation}.
\newblock {\em Inventiones mathematicae}, 225:529--602, 2021.

\bibitem{CL}
Jianchun Chu and Man-Chun Lee.
\newblock {Hypercritical deformed Hermitian-Yang-Mills equation revisited}.
\newblock {\em arXiv preprint arXiv:2206.00387}, 2022.

\bibitem{CLT}
Jianchun Chu, Man-Chun Lee, and Ryosuke Takahashi.
\newblock {A Nakai-Moishezon type criterion for supercritical deformed
  Hermitian-Yang-Mills equation}.
\newblock {\em arXiv preprint arXiv:2105.10725}, 2021.

\bibitem{CXY}
T~Collins, Dan Xie, and Shing-Tung Yau.
\newblock {The deformed Hermitian--Yang--Mills equation in geometry and
  physics}.
\newblock {\em Geometry and physics}, 1:69--90, 2018.

\bibitem{CJY}
Tristan~C Collins, Adam Jacob, and Shing-Tung Yau.
\newblock {$(1, 1) $ forms with specified Lagrangian phase: a priori estimates
  and algebraic obstructions}.
\newblock {\em Cambridge Journal of Mathematics}, 8(2):407--452, 2020.

\bibitem{CY}
Tristan~C Collins and Shing-Tung Yau.
\newblock {Moment maps, nonlinear PDE, and stability in mirror symmetry}.
\newblock {\em arXiv preprint arXiv:1811.04824}, 2018.

\bibitem{DP}
Ved~V Datar and Vamsi~Pritham Pingali.
\newblock {A numerical criterion for generalised Monge-Amp\`ere equations on
  projective manifolds}.
\newblock {\em Geometric and Functional Analysis}, 31(4):767--814, 2021.

\end{thebibliography}

\end{document}